\documentclass[11pt,reqno]{amsart}

\usepackage{amscd,amssymb,amsmath,amsthm}
\usepackage[arrow,matrix]{xy}
\usepackage{graphicx}
\usepackage{epstopdf}
\usepackage{color}
\newtheorem{thm}{Theorem}

\newtheorem{lemma}{Lemma}

\newtheorem{rk}{Remark}

\numberwithin{equation}{section} 

\def\r{\rho}

\def\C{\mathbb C}

\def\C{\mathbb{C}}

\begin{document}

\title[$p$-adic dynamical systems of the function $a x^{-2}$]{$p$-adic dynamical systems of the function $a x^{-2}$}

\author{U.A. Rozikov}

 \address{U.\ A.\ Rozikov \begin{itemize}
 \item[] V.I.Romanovskiy Institute of Mathematics of Uzbek Academy of Sciences;
\item[] Faculty of Mathematics, National University of Uzbekistan,
4, University str., 100174, Tashkent, Uzbekistan and
\item[] AKFA University, 1st Deadlock 10, Kukcha Darvoza, 100095, Tashkent, Uzbekistan.
\end{itemize}
} \email
{rozikovu@yandex.ru}

\begin{abstract} In this paper we study $p$-adic dynamical systems generated by
the function $f(x)={a\over x^2}$ in the set of complex $p$-adic numbers.
We find an explicit formula for the $n$-fold composition of $f$ for any $n\geq 1$.
Using this formula we give fixed points, periodic points, basin of attraction
and Siegel disk  of each fixed (periodic) point depending on parameters $p$ and $a$.
\end{abstract}

 \maketitle

\section{Introduction}

Nowadays the theory of $p$-adic numbers is one of very actively developing area in mathematics.
Because it has numerous applications in many branches of mathematics,
biology, physics and other sciences (see for example \cite{KJ}, \cite{Rpa}, \cite{VV} and the references therein).

In this paper we continue our study of $p$-adic dynamical systems generated by rational
functions (see \cite{ARS}-\cite{S} and references therein for motivations and history of $p$-adic dynamical systems).

Let us recall the main definition of the paper:

{\bf $p$-adic numbers}. Denote by
$(n,m)$ the greatest common
divisor of the positive integers $n$ and $m$.

Let $\mathbb{Q}$ be the field of rational numbers.

For each fixed prime number $p$, every rational number $x\neq 0$ can be represented in the
form $x=p^r\frac{n}{m}$, where $r,n\in\mathbb{Z}$, $m$ is a
positive integer, $(p,n)=1$, $(p,m)=1$.

The $p$-adic norm of $x$ is given by
$$
|x|_p=\left\{
\begin{array}{ll}
p^{-r}, & \ \textrm{ for $x\neq 0$},\\[2mm]
0, &\ \textrm{ for $x=0$}.\\
\end{array}
\right.
$$
It has the following properties:

1) $|x|_p\geq 0$ and $|x|_p=0$ if and only if $x=0$,

2) $|xy|_p=|x|_p|y|_p$,

3) the strong triangle inequality
$$
|x+y|_p\leq\max\{|x|_p,|y|_p\},
$$

3.1) if $|x|_p\neq |y|_p$ then $|x+y|_p=\max\{|x|_p,|y|_p\}$,

3.2) if $|x|_p=|y|_p$ then for $p=2$ we have
$|x+y|_p\leq {1\over 2}|x|_p$ (see \cite{VV}).

The completion of $\mathbb{Q}$ with  respect to $p$-adic norm defines the
$p$-adic field which is denoted by $\mathbb{Q}_p$ (see \cite{Ko}).

The algebraic completion of $\mathbb{Q}_p$ is denoted by $\C_p$ and it is
called the set of {\it complex $p$-adic numbers}.

 For any $a\in\C_p$ and
$r>0$ denote
$$
U_r(a)=\{x\in\C_p : |x-a|_p<r\},\ \ V_r(a)=\{x\in\C_p :
|x-a|_p\leq r\},
$$
$$
S_r(a)=\{x\in\C_p : |x-a|_p= r\}.
$$

{\bf Dynamical systems in $\C_p$.}
To define a dynamical system we consider a function  $f: x\in U\to f(x)\in U$, (in this paper
 $U=U_r(a)$ or $\C_p$) (see for example \cite{PJS}).

For $x\in U$ denote
by $f^n(x)$ the $n$-fold composition of $f$ with itself (i.e. $n$ time iteration of $f$ to $x$):
$$f^n(x)=\underbrace{f(f(f\dots (f}_{n \,{\rm times}}(x)))\dots).$$

For arbitrary given $x_0\in U$ and $f:U\to U$ the discrete-time
dynamical system (also called the trajectory) of $x_0$ is the sequence of points
\begin{equation}\label{eds}
x_0, x_1=f(x_0), x_2=f^2(x_0), x_3=f^3(x_0), \dots
\end{equation}
{\it The main problem:} Given a function $f$ and initial point $x_0$ what ultimately happens with the sequence (\ref{eds}).
Does the limit $\lim_{n\to\infty} x_n$ exist? If not what is the set of limit points of the sequence?

A point $x\in U$ is called a fixed point for $f$ if $f(x)=x$. The point $x$
is a periodic point of period $m$ if $f^m(x) = x$. The least positive $m$ for which
$f^m(x) = x$ is called the prime period of $x$.

A fixed point $x_0$ is called an {\it
attractor} if there exists a neighborhood $U(x_0)$ of $x_0$ such
that for all points $x\in U(x_0)$ it holds
$\lim\limits_{n\to\infty}f^n(x)=x_0$. If $x_0$ is an attractor
then its {\it basin of attraction} is
$$
\mathcal A(x_0)=\{x\in \C_p :\ f^n(x)\to x_0, \ n\to\infty\}.
$$
A fixed point $x_0$ is called {\it repeller} if there  exists a
neighborhood $U(x_0)$ of $x_0$ such that $|f(x)-x_0|_p>|x-x_0|_p$
for $x\in U(x_0)$, $x\neq x_0$.

Let $x_0$ be a fixed point of a
function $f(x)$.
Put $\lambda=f'(x_0)$. The point $x_0$ is attractive if $0<|\lambda|_p < 1$, {\it indifferent} if $|\lambda|_p = 1$,
and repelling if $|\lambda|_p > 1$.

The ball $U_r(x_0)$ (contained in $V$) is said to
be a {\it Siegel disk} if each sphere $S_{\r}(x_0)$, $\r<r$ is an
invariant sphere of $f(x)$, i.e. if $x\in S_{\r}(x_0)$ then all
iterated points $f^n(x)\in S_{\r}(x_0)$ for all $n=1,2\dots$.  The
union of all Siegel desks with the center at $x_0$ is called {\it
a maximum Siegel disk} and is denoted by $SI(x_0)$.

In  Section 2 we consider the function $f(x)={a\over x^2}$ and study the dynamical systems generated by this function in $\C_p$.
We give fixed points, periodic points, basin of attraction and Siegel disk of each fixed (and periodic) point.

\section{The function $a/x^2$}

Consider the dynamical system associated with the function $f:\C_p\to\C_p$ defined by
\begin{equation}\label{fz}
f(x)=\frac{a}{x^2}, \ \ a\neq 0, \ \  a\in \C_p,
\end{equation}
where $x\ne 0$.

Denote by $\theta_{j,n}$, $j = 1,\dots, n$, the $n$th root of unity in $\C_p$, while $\theta_{1,n} = 1$.

This function has three  fixed point $x_k$, $k=1,2,3$,  which are solutions to
$x^3=a$ in $\C_p$.

For these fixed points we have
\begin{equation}\label{nor}
x^3_k=a \ \ \Rightarrow \ \ x_k=\theta_{k,3}a^{1\over 3} \ \ \Rightarrow \ \ |x^3_k|_p=|a|_p \ \ \Rightarrow \ \ |x_k|_p=\alpha\equiv(|a|_p)^{1/3}.
\end{equation}
Thus $x_k\in S_\alpha(0)$, $k=1,2,3$.

We have
$$f'(x)={-2a\over x^3}= {-2\over x}\cdot f(x).$$
Therefore at a fixed point we get
$$f'(x_k)={-2\over x_k}\cdot f(x_k)=-2.$$
$$|f'(x_k)|_p=\left\{\begin{array}{ll}
1/2, \ \ \mbox{if} \ \ p=2\\[2mm]
1, \ \ \mbox{if} \ \ p\geq 3
\end{array}\right.$$
Hence the fixed point $x_k$ is an attractive for $p=2$ and an indifferent for $p\geq 3$. Therefore the fixed point is never repeller.

We can explicitly calculate $f^n$:

\begin{lemma}\label{fo} For any $x\in \C_p\setminus \{0\}$ we have
$$f^n(x)=a^{{1\over 3}(1-(-2)^n)}\cdot x^{(-2)^n}, \ \ n\geq 1.$$
\end{lemma}
\begin{proof} We use induction over $n$.
For $n=1,2$ the formula is clear. Assume it is true for $n$ and show it for $n+1$:
$$f^{n+1}(x)=f^n(f(x))=a^{{1\over 3}(1-(-2)^n)}\cdot (f(x))^{(-2)^n}$$
$$= a^{{1\over 3}(1-(-2)^n)}\cdot ({a\over x^2})^{(-2)^n}=a^{{1\over 3}(1-(-2)^{n+1})}\cdot x^{(-2)^{n+1}}.$$
This completes the proof.
\end{proof}
Recall $\alpha=(|a|_p)^{1/3}$. For $r>0$, take $x\in S_r(0)$, i.e.,  $|x|_p=r$.
Then we have

\begin{equation}\label{na}
|f^n(x)|_p=\left|a^{{1\over 3}(1-(-2)^n)}\cdot x^{(-2)^n}\right|_p
=\alpha^{1-(-2)^n}\cdot r^{(-2)^n}, \ \ n\geq 1.
\end{equation}

\subsection{Dynamics on $\C_p\setminus S_\alpha(0)$}

\begin{lemma} For $\alpha$ defined in (\ref{nor}) the following assertions hold
\begin{itemize}
\item[1.] The sphere $S_\alpha(0)$ is invariant with respect to $f$, (i.e., $f(S_\alpha(0))\subset S_\alpha(0)$);
\item[2.] $f(U_\alpha(0))\subset \C_p\setminus V_\alpha(0)$;
\item[3.] $f(\C_p\setminus V_\alpha(0))\subset U_\alpha(0)$.
\end{itemize}
\end{lemma}
\begin{proof}
1. If $x\in S_\alpha(0)$, i.e., $|x|_p=\alpha$, then
$$|f(x)|_p=|\frac{a}{x^2}|_p={|a|_p\over \alpha^2}=\alpha.$$

2. If $x\in U_\alpha(0)$, i.e., $|x|_p<\alpha$, then
$$|f(x)|_p=|\frac{a}{x^2}|_p>{|a|_p\over \alpha^2}=\alpha.$$
Therefore, $f(x)\in \C_p\setminus V_\alpha(0)$.
Proof of the part 3 is similar.
\end{proof}

\begin{lemma} The function (\ref{fz}) does not have any periodic point in $\C_p\setminus S_\alpha(0)$.
\end{lemma}
\begin{proof} We know that all three fixed points belong to $S_\alpha(0)$.
Let $x\in \C_p\setminus S_\alpha(0)$ be a $m$-periodic ($m\geq 2$) point for (\ref{fz}), i.e.,
$x$ satisfies $f^m(x)=x$. Then it is necessary that $|f^m(x)|_p=|x|_p$.
But for any $x\in \C_p\setminus S_\alpha(0)$ (i.e. $|x|_p=r\ne \alpha$),
by (\ref{na}) we get
\begin{equation}\label{nap}
|f^m(x)|_p=\alpha^{1-(-2)^m}\cdot r^{(-2)^m}=\alpha \cdot \left(r\over \alpha\right)^{(-2)^m}\ne r, \ \ \forall r\ne \alpha.
\end{equation}
Therefore, $f^m(x)=x$ is not satisfied for any $x\in \C_p\setminus S_\alpha(0)$.
\end{proof}

For given $r>0$, denote
$$r_n= \alpha^{1-(-2)^n}\cdot r^{(-2)^n}.$$
Then by (\ref{na}) one can see that the trajectory $f^n(x)$, $n\geq 1$ of $x\in S_r(0)$
has the following sequence of spheres on its route:
$$S_r(0) \rightarrow  S_{r_1}(0)\rightarrow S_{r_2}(0)\rightarrow S_{r_3}(0)\rightarrow\dots$$

Now we calculate the limits of $r_n$.

{\bf Case of even $n$}.  From (\ref{na}) it is easy to see that

$$
\lim_{n\to \infty} |f^n(x)|_p=\lim_{n\to \infty} r_n=\left\{\begin{array}{lll}
0, \ \ \mbox{if} \ \ r<\alpha\\[2mm]
\alpha, \ \ \mbox{if} \ \ r=\alpha\\[2mm]
+\infty, \ \ \mbox{if} \ \ r>\alpha
\end{array}
\right.
$$

{\bf Case of odd $n$}. In this case we have

$$
\lim_{n\to \infty} |f^n(x)|_p=\lim_{n\to \infty} r_n=\left\{\begin{array}{lll}
+\infty, \ \ \mbox{if} \ \ r<\alpha\\[2mm]
\alpha, \ \ \mbox{if} \ \ r=\alpha\\[2mm]
0, \ \ \mbox{if} \ \ r>\alpha
\end{array}
\right.
$$

Summarizing above-mentioned results we obtain
the following theorem:

\begin{thm}\label{ts} If $\alpha$ is defined by (\ref{nor}). Then
\begin{itemize}
\item[1.] if $x\in U_\alpha(0)$ then
$$\lim_{k\to \infty}f^{2k}(x)=0, \ \ \ \lim_{k\to \infty}|f^{2k-1}(x)|_p=+\infty.$$
\item[2.] if $x\in S_\alpha(0)$ then  $f^n(x)\in S_\alpha(0), \, n\geq 1.$
\item[3.] if $x\in \C_p\setminus V_\alpha(0)$ then
 $$\lim_{k\to \infty}|f^{2k}(x)|_p=+\infty, \ \ \ \lim_{k\to \infty}f^{2k-1}(x)=0.$$
 \end{itemize}
\end{thm}
\begin{rk} Note that Theorem \ref{ts} is true for more general function: $f(x)={a\over x^q}$, where $q$ is a natural number, $q\geq 2$.
In this case $\alpha=|a|_p^{1/(q+1)}$. The case $q=1$ is simple: in this case any point $x\in \C_p\setminus \{0\}$ is 2-periodic.
That is $f(f(x))=x$. Indeed,
$$f(f(x))={a\over {a\over x}}=a\cdot {x\over a}=x.$$
\end{rk}
\subsection{Dynamics on $S_\alpha(0)$.}
By Theorem \ref{ts} it remains to study the dynamical system of $f:S_\alpha(0)\to S_\alpha(0)$.
Recall that all fixed points $x_k$, $k=1,2,3$ are in $S_\alpha(0)$.

\begin{lemma} The distance between fixed points is
\begin{equation}\label{df}
|x_1-x_2|_p= |x_1-x_3|_p=|x_2-x_3|_p=\left\{\begin{array}{ll}
\alpha, \ \ \mbox{if} \ \ p\ne 3\\[2mm]
{\alpha\over \sqrt{3}},  \ \ \mbox{if} \ \ p=3
\end{array}
\right.
\end{equation}
\end{lemma}
\begin{proof} Since $x_i^3=a$, $i=1,2,3$, for $x_i\ne x_j$ we have
$$0=x_i^3-x_j^3=(x_i-x_j)(x_i^2+x_ix_j+x_j^2) \ \ \Rightarrow \ \ x_i^2+x_ix_j+x_j^2=0$$
$$ \Leftrightarrow \ \ (x_i-x_j)^2=-3x_ix_j \ \ \Rightarrow \ \ |x_i-x_j|_p^2=|3x_ix_j|_p.$$
From the last equality, using $|x_i|_p=|x_j|_p=\alpha$, we get (\ref{df}).
\end{proof}
Take $x\in S_\alpha(0)$ such that $|x-x_1|_p=\rho$, i.e.,
$x=x_1+\gamma$, with $|\gamma|_p=\rho$. Note that $\rho\leq \alpha$. Then by Lemma \ref{fo} we have
\begin{equation}\label{nf}
|f^n(x)-x_1|_p= |f^n(x)-f^n(x_1)|_p=\alpha^{1-(-2)^n}|x^{(-2)^n}-x_1^{(-2)^n}|_p.
\end{equation}
Now we use the following formula
$$x^{2^n}-y^{2^n}=(x-y)\prod_{j=0}^{n-1}(x^{2^j}+y^{2^j}).$$
Then from (\ref{nf}) we get
\begin{equation}\label{2n}
|f^n(x)-x_1|_p= \alpha^{1-(-2)^n}\cdot \left\{\begin{array}{ll}
\rho\prod_{j=0}^{n-1}|(x_1+\gamma)^{2^j}+x_1^{2^j}|_p, \ \ \mbox{if} \ \ n \ \ \mbox{is even}\\[3mm]
{\rho\over |xx_1|_p}\prod_{j=0}^{n-1}|(x_1+\gamma)^{-2^j}+x_1^{-2^j}|_p, \ \ \mbox{if} \ \ n \ \ \mbox{is odd}.
\end{array}\right.
\end{equation}

We have
\begin{equation}\label{3a}
|(x_1+\gamma)^{2^j}+x_1^{2^j}|_p=\left|2x_1^{2^j}+\sum_{s=1}{2^j\choose s}x_1^{2^j-s}\gamma^s\right|_p=\left\{\begin{array}{ll}
|2|_p\alpha^{2^j}, \ \ \mbox{if} \ \ \rho<\alpha\\[2mm]
\leq |2|_p\alpha^{2^j}, \ \ \mbox{if} \ \ \rho=\alpha.
\end{array}
\right.
\end{equation}
Here we used that
$$\left|{2^j\choose s}\right|_p\leq \left\{\begin{array}{ll}
{1\over 2}, \ \ \mbox{if} \ \ p=2\\[2mm]
1, \ \ \mbox{if} \ \ p\geq 3
\end{array}\right.$$

Using (\ref{3a}) we get
\begin{equation}\label{3b}
|(x_1+\gamma)^{-2^j}+x_1^{-2^j}|_p= {|(x_1+\gamma)^{2^j}+x_1^{2^j}|_p \over   |(x_1+\gamma)^{2^j}x_1^{2^j}|_p}=  \left\{\begin{array}{lll}
|2|_p\alpha^{-2^j}, \ \ \mbox{if} \ \ \rho<\alpha\\[2mm]
\leq |2|_p {1\over |(x_1+\gamma)^{2^j}|_p}, \ \ \mbox{if} \ \ \rho=\alpha.
\end{array}
\right.
\end{equation}
In \textbf{case of even} $n$, by (\ref{3a}) from (\ref{2n}) we get
$$
|f^n(x)-x_1|_p= \alpha^{1-2^n}\cdot \rho\prod_{j=0}^{n-1}|(x_1+\gamma)^{2^j}+x_1^{2^j}|_p
$$
\begin{equation}\label{2na}
=\rho\cdot
\alpha^{1-2^n}\cdot |2|_p^n \prod_{j=0}^{n-1}\alpha^{2^j}\cdot\left\{\begin{array}{ll}
1, \ \ \mbox{if} \ \ \rho<\alpha\\[2mm]
\leq 1, \ \ \mbox{if} \ \ \rho=\alpha
\end{array}\right. =\rho\cdot |2|_p^n\cdot\left\{\begin{array}{ll}
1, \ \ \mbox{if} \ \ \rho<\alpha\\[2mm]
\leq 1, \ \ \mbox{if} \ \ \rho=\alpha
\end{array}\right.
\end{equation}

Similarly, in \textbf{case of odd} $n$, by (\ref{3b}) from (\ref{2n}) we get
\begin{equation}\label{2na}
|f^n(x)-x_1|_p= \alpha^{1+2^n}\cdot {\rho\over \alpha^2}\cdot |2|_p^n\prod_{j=0}^{n-1}\alpha^{-2^j}
 =\rho \cdot |2|_p^n\ \ \mbox{if} \ \ \rho<\alpha.
\end{equation}
The same formulas are also true for $x_2$ and $x_3$.

For fixed $\alpha$ (defined in (\ref{nor})) and $t\in S_\alpha(0)$ denote
$$\mathcal S_{\rho,t}=S_\alpha(0)\cap S_\rho(t)=\{x\in S_\alpha(0): |x-t|_p=\rho\}.$$
Thus we have proved the following lemma
\begin{lemma}\label{sw} Let  $\rho<\alpha$.
 Then for any $x\in \mathcal S_{\rho, x_i}$ ($i=1,2,3$) we have
 \begin{itemize}
 \item if $p=2$ then
 $$f^n(x)\in \mathcal S_{2^{-n}\rho, x_i}.$$
 \item if $p\geq 3$ then
$$f^n(x)\in \mathcal S_{\rho, x_i}, \ \ n\geq 1.$$
In particular, the set $\mathcal S_{\rho, x_i}$ is invariant with respect to $f$ for any $\rho< \alpha$.
\end{itemize}
\end{lemma}
Denote
$$\mathcal V_{\rho, t}=\bigcup_{0\leq r< \rho}\mathcal S_{r,t}=\{x\in S_\alpha(0): |x-t|_p<\rho\}.$$
\begin{lemma}\label{3d} If $x\in \mathcal S_{\rho, x_i}$, for some $i=1,2,3$ then
\begin{itemize}
\item[i.]
If $\rho$ is such that
$$\rho<\left\{\begin{array}{ll}
\alpha, \ \ \mbox{if} \ \ p\ne 3\\[2mm]
{\alpha\over \sqrt{3}},  \ \ \mbox{if} \ \ p=3.
\end{array}\right.
$$
 then
$$x\in \left\{\begin{array}{ll}
\mathcal S_{{\alpha\over \sqrt{3}}, x_j}, \ \ \mbox{for} \ \ p=3\\[2mm]
\mathcal S_{\alpha, x_j}, \ \ \mbox{for} \ \ p\ne 3,
\end{array}\right. \ \ j\ne i.$$
\item[ii.] If $p=3$ and $\rho\geq {\alpha\over \sqrt{3}}$ then
$$x\in \left\{\begin{array}{ll}
\mathcal V_{\rho, x_j}, \ \ \mbox{for} \ \ \rho= {\alpha\over \sqrt{3}}\\[2mm]
\mathcal S_{\rho, x_j}, \ \ \mbox{for} \ \ \rho> {\alpha\over \sqrt{3}},
\end{array}\right. \ \ j\ne i.$$
\end{itemize}
\end{lemma}
\begin{proof} For $x\in \mathcal S_{\rho, x_i}$, using property of $p$-adic norm and formula (\ref{df}) we get
$$|x-x_j|_p=|x-x_i+x_i-x_j|_p=\left\{\begin{array}{llll}
\alpha, \ \ \mbox{if} \ \ p\ne 3\\[2mm]
{\alpha\over \sqrt{3}},  \ \ \mbox{if} \ \ p=3, \ \ \rho<{\alpha\over \sqrt{3}}\\[2mm]
\leq\rho,  \ \ \mbox{if} \ \ p=3, \ \ \rho={\alpha\over \sqrt{3}}\\[2mm]
\rho,  \ \ \mbox{if} \ \ p=3, \ \ \rho>{\alpha\over \sqrt{3}}
\end{array}
\right. $$
This completes the proof.
\end{proof}
Denote
$$\mathcal U_\alpha=\{x\in S_\alpha(0): |x-x_1|_p=|x-x_2|_p=|x-x_3|_p=\alpha\}.$$
As a corollary of Lemma \ref{3d} we have
\begin{lemma} If $p\ne 3$ then $S_\alpha(0)$ has the following partition
$$S_\alpha(0)=\mathcal U_\alpha\cup \bigcup_{i=1}^3\mathcal V_{\alpha, x_i}.$$
\end{lemma}
\begin{lemma} Let $\alpha$ is defined by (\ref{nor}) then
\begin{itemize}
\item[1.] If $p=2$ then the set $\mathcal U_\alpha$ is invariant with respect to $f$.

\item[2.] If $p\geq 3$ and $x\in \mathcal U_\alpha$ then one of the following assertions holds

\item[2.a)] There exists $n_0$ and $\mu_{n_0}< \alpha$ such that
$$\begin{array}{ll}
f^n(x)\in \mathcal U_\alpha, \ \ \forall n\leq n_0,\\[2mm]
f^n(x)\in \mathcal S_{\mu_{n_0}}(x_i), \ \ \forall n>n_0\ \ \mbox{for some} \ \ i=1,2,3.
\end{array}$$
\item[2.b)] $f^n(x)\in \mathcal U_\alpha, \ \ \forall n\geq 1.$
\end{itemize}
\end{lemma}
\begin{proof}
1. For any $x\in \mathcal U_\alpha$ we have
$$|f(x)-x_i|_p=\left|{a\over x^2}-{a\over x_i^2}\right|_p=|a|_p\left|{(x_i-x)(x_i+x)\over x^2x_i^2}\right|_p$$
\begin{equation}\label{p2}
=\alpha^3\cdot {\alpha |x+x_i|_p\over \alpha^4}=|x+x_i|_p=|x-x_i+2x_i|_p=\left\{\begin{array}{ll}
\alpha, \ \ \mbox{if} \ \ p=2\\[2mm]
\mu_{1,i}, \ \ \mbox{if} \ \ p\geq 3,
\end{array}\right.\end{equation}
where $\mu_{1,i}\leq \alpha.$ The part 1 follows from this equality.

2. If in (\ref{p2}) there exists $i$ such that $\mu_{1,i}=|x+x_i|_p<\alpha$, then
 $f(x)\in \mathcal S_{\mu_{1,i}, x_i}$. The set $\mathcal S_{\mu_{1,i}, x_i}$ is invariant with respect to $f$.
In case of all $\mu_{1,i}=\alpha$ we have $f(x)\in \mathcal U_\alpha$. Then we note that
$$|f^2(x)-x_i|_p=|f(x)-x_i+2x_i|_p=\left\{\begin{array}{ll}
\alpha, \ \ \mbox{if} \ \ p=2\\[2mm]
\mu_{2,i}\leq \alpha, \ \ \mbox{if} \ \ p\geq 3
\end{array}\right.$$
Thus we can repeat the above argument: if there exists $i$ such that $\mu_{2,i}<\alpha$, then  $f^2(x)\in \mathcal S_{\mu_{2,i}, x_i}$ which is invariant with respect to $f$. If all $\mu_{2,i}=\alpha$ then $f^2(x)\in \mathcal U_\alpha$. Iterating this argument one proves the part 2.
\end{proof}

\begin{lemma} For $k\in \{1,2,3\}$, $j\in \{1,2,3\}\setminus \{k\}$ and fixed points $x_k$, $x_j$ we have
\begin{itemize}
\item[1.] $x_j\notin \mathcal V_{\rho, x_k},$
if and only if
$$\rho\leq\left\{\begin{array}{ll}
\alpha, \ \ \mbox{if} \ \ p\ne 3\\[2mm]
{\alpha\over \sqrt{3}},  \ \ \mbox{if} \ \ p=3.
\end{array}
\right.$$
\item[2.] if $p=2$ then
$$\mathcal V_{\alpha, x_j}\cap \mathcal V_{\alpha, x_k}=\emptyset, \ \ \mbox{for all} \ \ j, k\in \{1,2,3\}, \, j\ne k,$$
\end{itemize}
\end{lemma}
\begin{proof} Follows from (\ref{df}) and Lemma \ref{sw}.
\end{proof}

Summarizing above mentioned results we get
\begin{thm}\label{tox}
If $\alpha$ is defined by (\ref{nor}). Then for the dynamical system generated by $f: S_\alpha(0)\to S_\alpha(0)$ given in (\ref{fz}) the following assertions hold.
\begin{itemize}
\item[1.] if $p=2$ then $\mathcal A(x_j)= \mathcal V_{\alpha, x_j}$, i.e.,
$$\lim_{n\to \infty}f^{n}(x)=x_j, \ \ \mbox{for any} \ \ x\in  \mathcal V_{\alpha, x_j}.$$
$$f^n(x)\in \mathcal U_\alpha, \ \ n\geq 1, \ \ \mbox{for all} \ \ x\in \mathcal U_\alpha.$$
\item[2.] if $p\geq 3$ then
$$SI(x_j)=\mathcal V_{\alpha, x_j}, \ \ j\in \{1,2,3\}.$$
Moreover,
$$SI(x_1)=SI(x_2)=SI(x_3), \ \ \mbox{if} \ \ p=3.$$
$$SI(x_j)\cap SI(x_k)=\emptyset, \ \ \mbox{if} \ \ p>3.$$
\item[3.] If $p\geq 3$ and $x\in \mathcal U_\alpha$ then one of the following assertions holds

\item[3.a)] There exists $n_0$ and $\mu_{n_0}< \alpha$ such that
$$\begin{array}{ll}
f^n(x)\in \mathcal U_\alpha, \ \ \forall n\leq n_0,\\[2mm]
f^n(x)\in \mathcal S_{\mu_{n_0}}(x_i), \ \ \forall n>n_0\ \ \mbox{for some} \ \ i=1,2,3.
\end{array}$$
\item[3.b)] $f^n(x)\in \mathcal U_\alpha, \ \ \forall n\geq 1.$
 \end{itemize}
\end{thm}

This theorem does not give behavior of $f^n(x)\in \mathcal U_\alpha, \ \ n\geq 1$, i.e., in the case when the trajectory remains
  in $\mathcal U_\alpha$ (that is when $p=2$ and in the case part 3.b of Theorem \ref{tox}).
  Since there is not any fixed point of $f$ in $\mathcal U_\alpha$, below we are interested
   to periodic points of $f$ in $\mathcal U_\alpha$: for a given natural $m\geq 2$ the $m$-periodic
   points of this set are solutions of
   the following system of equations
\begin{equation}\label{pp}\begin{array}{ll}
f^m(x)=a^{{1\over 3}(1-(-2)^m)}\cdot x^{(-2)^m}=x,\\[2mm]
|x-x_1|_p=|x-x_2|_p=|x-x_3|_p=\alpha.
\end{array}
\end{equation}
\begin{rk} Note that in case $m=2$, there is no any solution to the first equation of (\ref{pp}) (except fixed points).
Therefore below we consider the case $m\geq 3$.
\end{rk}
Denote
$$M_m=\left\{\begin{array}{ll}
\left\{(j,p): \left|\theta_{k,3}-\theta_{j, 2^m-1}\right|_p=1, \ \ \forall k=1,2,3\right\} \ \ \mbox{if} \ \ m \ \ \mbox{is even}\\[2mm]
\left\{(j,p): \left|\theta_{k,3}-\theta_{j, 2^m+1}\right|_p=1, \ \ \forall k=1,2,3\right\} \ \ \mbox{if} \ \ m \ \ \mbox{is odd}
\end{array}
\right.$$
\begin{lemma} The solutions of the system (\ref{pp}) in $\C_p$ are
\begin{equation}\label{ss}
\hat x_j=a^{1\over 3}\cdot \left\{\begin{array}{ll}
\theta_{j, 2^m-1}, \ \ \mbox{if} \ \ m \ \ \mbox{is even}\\[2mm]
1/\theta_{j, 2^m+1}, \ \ \mbox{if} \ \ m \ \ \mbox{is odd}
\end{array}
\right.
\end{equation}
where $(j,p)\in M_m$.
\end{lemma}
\begin{proof} From (\ref{pp}) we get
$$\left({x\over a^{1/3}}\right)^{(-2)^m-1}=1.$$
Which has solutions (\ref{ss}). The condition $(j,p)\in M_m$ is needed to satisfy the second equation of the system (\ref{pp}).
\end{proof}
\begin{rk} We note that 
\begin{itemize}
\item
in the case $p=2$,  by part 1 of Theorem \ref{tox}, it follows that all $m$-periodic points (except fixed ones)
mentioned in (\ref{ss}) belong to $\mathcal U_\alpha$.

\item in the case $m\geq 3$ and $p\geq 3$ it is not clear to see $M_m\ne \emptyset$.  It is known that (see \cite[Corollary 2.2.]{AT})  the equation $x^k = 1$ has $g = (k, p - 1)$ different roots in $\mathbb Q_p$.
Using this fact and assuming that $a\in \mathbb Q_p$ and $a^{1\over 3}$ exists in $\mathbb Q_p$,  one can see how many periodic solutions of (\ref{pp})
exist in $\mathbb Q_p$. For example, if  $p=31$ then $t^3=1$ (with $t={x\over a^{1/3}}$) has $g=(3, 30)=3$, i.e., all possible solutions in $\mathbb Q_p$ and for $m=4$ the equation $t^{2^4-1}=1$ has $g=(15,30)=15$ distinct solutions in  $\mathbb Q_p$. Three of 15 solutions coincide with solutions of $t^3=1$, therefore remains 12 distinct solutions to satisfy the second equation of (\ref{pp}). For these solution one can check the condition $M_m\ne \emptyset$.
\end{itemize}
\end{rk}

\begin{lemma}\label{cc} If $x_*$ is a solution to (\ref{pp}) then
$$x_* \ is \ \left\{\begin{array}{ll}
{\rm attracting}, \ \  \mbox{if} \ \ p=2\\[2mm]
{\rm indifferent}, \ \ \mbox{if} \ \ p\geq 3.
\end{array}
\right.$$
\end{lemma}
\begin{proof} We have
$$\left|(f^m)'(x_*)\right|_p=\left|(-2)^m\cdot a^{{1\over 3}(1-(-2)^m)}\cdot x_*^{(-2)^m-1}\right|_p
$$ $$ =\left|(-2)^m\cdot {f^m(x_*)\over x_*}\right|_p=\left\{\begin{array}{ll}
1/2^m, \ \ \mbox{if} \ \ p=2\\[2mm]
1, \ \ \mbox{if} \ \ p\geq 3
\end{array}\right.$$
This completes the proof.
\end{proof}
Consider a $m$-periodic point $x_*$. It is clear that this point is a fixed point for the function $\varphi(x)\equiv f^m(x)$.
The point $x_*$ generates $m$-cycle:
$$x_*, x^{(1)}=f(x_*), \dots, x^{(m-1)}=f^{m-1}(x_*).$$
Clearly, each element of this cycle is fixed point for function $\varphi$.
We use the following
\begin{thm}\label{at}\cite{AT}  Let $x_0$ be a fixed point of an analytic function $\varphi:U\to U$.
The following assertions hold
\begin{itemize}
\item[1.] if $x_0$ is an attractive point of $\varphi$ and if $r>0$ satisfies the inequality
$$
Q=\max_{1\leq n<\infty}\bigg|\frac{1}{n!}\frac{d^n\varphi}{dx^n}(x_0)\bigg|_pr^{n-1}<1
$$
and $U_r(x_0)\subset U$ then $U_r(x_0)\subset \mathcal A(x_0)$;\\

\item[2.] if $x_0$ is an indifferent point of $\varphi$ then it is the center of a Siegel disk. If $r$
satisfies the inequality
$$
S=\max_{2\leq n<\infty}\bigg|\frac{1}{n!}\frac{d^n\varphi}{dx^n}(x_0)\bigg|_pr^{n-1}<|\varphi'(x_0)|_p
$$
and $U_r(x_0)\subset U$ then $U_r(x_0)\subset SI(x_0)$;
\end{itemize}
\end{thm}

Lemma \ref{cc} suggests the following
\begin{thm} \begin{itemize}
\item If $p=2$ then for any $m = 2, 3,\dots$, the $m$-cycles are attractors and open balls with 
radius $\alpha$ are contained in the basins of attraction.
\item If $p\geq 3$ then for any $m = 2, 3,\dots$, every $m$-cycle is a center of a Siegel
disk with radius $\alpha$.
\end{itemize}
\end{thm}
\begin{proof} Let $x_*$ be a $m$-periodic point. Recall that $|x_*|_p=\alpha$.  We use Theorem \ref{at}, by Lemma \ref{fo} we get: 
$$
Q=\max_{1\leq n<\infty}\bigg|\frac{1}{n!}\frac{d^n\varphi}{dx^n}(x_*)\bigg|_pr^{n-1}
= \max_{1\leq n<\infty}\bigg|\frac{1}{n!}a^{{1\over 3}(1-(-2)^m)}\cdot \prod_{s=0}^{n-1}\left((-2)^{m}-s\right)\cdot   x_*^{(-2)^m-n}\bigg|_pr^{n-1}
$$
$$
= \max_{1\leq n<\infty}\bigg|\frac{1}{n!}\cdot \prod_{s=0}^{n-1}\left((-2)^{m}-s\right)\cdot   {x_*\over x_*^n}\bigg|_pr^{n-1}$$
$$= \max_{1\leq n<\infty}\bigg|\frac{1}{n!}\cdot \prod_{s=0}^{n-1}\left((-2)^{m}-s\right)\bigg|_p\left({r\over \alpha}\right)^{n-1}$$
$$=\max_{1\leq n<\infty}\left({r\over \alpha}\right)^{n-1}\cdot\left\{
\begin{array}{ll}\left|{2^m\choose n}\right|_p, \ \ \mbox{if} \ \ m-{\rm even}\\[2mm]
\left|{2^m+n\choose 2^m}\right|_p, \ \ \mbox{if} \ \ m-{\rm odd}
\end{array}\right.<1
$$
If $r <\alpha$, this condition is satisfied. The second part is similar.
\end{proof} 


\end{document}